\documentclass[12pt]{amsart}
\usepackage{cite}
\usepackage{graphicx}
\usepackage{amsmath, amsfonts, amssymb,latexsym, graphics, graphicx, mathrsfs, manfnt, textcomp, appendix, chemarrow}
\usepackage{hyperref}
\usepackage[all]{xy}

\author{Elif Uyan\i k}
\address{Elif Uyan\i k, Department of Mathematics \\ Middle East Technical University \\ 06800 Ankara Turkey}
\email{euyanik@metu.edu.tr}
\author{Murat H. Yurdakul}
\address{Murat H. Yurdakul, Department of Mathematics \\ Middle East Technical University \\ 06800 Ankara Turkey}
\email{myur@metu.edu.tr}

\title[]{A Remark on a paper of P. B. Djakov and M. S. Ramanujan}

\thanks{This research was partially supported by Turkish Scientific and Technological Research Council.}

\numberwithin{equation}{section}

\theoremstyle{thmit} 
\newtheorem{theorem}{Theorem}[section]

\newtheorem{corollary}[theorem]{Corollary}
\newtheorem{proposition}[theorem]{Proposition}

\begin{document}

\maketitle

\begin{abstract}
Let $\ell$ be a Banach sequence space with a monotone norm in which the canonical system $(e_{n})$ is an unconditional basis. We show that if there exists a continuous linear unbounded operator between $\ell$-K\"{o}the spaces, then there exists a continuous unbounded quasi-diagonal operator between them. Using this result, we study in terms of corresponding K\"{o}the matrices when every continuous linear operator between $\ell$-K\"{o}the spaces is bounded. As an application, we observe that the existence of an unbounded operator between $\ell$-K\"{o}the spaces, under a splitting condition, causes the existence of a common basic subspace.
\end{abstract}

\maketitle
\section{Introduction}

     Following \cite{Dra83}, we denote by $\ell$ a Banach sequence space in which the canonical system $(e_{n})$ is an unconditional basis. The norm $\left\|.\right\|$ is called monotone if $\left\|x\right\| \leq \left\|y\right\|$ whenever $\left|x_{n}\right| \leq \left|y_{n}\right|$, $x=(x_{n})$, $y=(y_{n}) \in \ell$, $n \in \mathbb{N}$. Let $\Lambda$ be the class of such spaces with monotone norm. In particular, $l_{p} \in \Lambda$ and $c_{0} \in \Lambda$. It is known that every Banach space with an unconditional basis $(e_{n})$ has a monotone norm which is equivalent to its original norm. Indeed, it is enough to put
		$$ 
		\left\|x\right\| = \sup_{\left|\beta_{n}\right| \leq 1} \left|\sum_{n} e_{n}{'}({x}) \beta_{n} e_{n}\right|
		$$
		where $\left|.\right|$ denotes the original norm, $(e_{n}{'})$ denote the sequence of coefficient functionals.
		
		Let $\ell \in \Lambda$ and $\left\|.\right\|$ be a monotone norm in $\ell$. If $A = (a_{n}^{k})$ is a K\"{o}the matrix, the $\ell$-K\"{o}the space $\lambda^{\ell}(A)$ is the space of all sequences of scalars $(x_{n})$ such that $(x_{n}a_{n}^{k}) \in \ell$ with the topology generated by the seminorms
		$$
		\left\|(x_{n})\right\|_{k} = \left\|(x_{n} a_{n}^{k})\right\|
		$$
		
		For any linear operator $T: X \longrightarrow Y$ between Fr{\'e}chet spaces we consider the following operator seminorms
		$$
		\left\|T\right\|_{p,q} = \sup \left\{ \left\|Tx\right\|_{p} : \left\|x\right\|_{q} \leq 1\right\}, \quad p,q \in \mathbb{N}
		$$
which may take the value $+\infty$. In particular, for any one dimensional operator $ T = u \otimes x$, we have
		$$
		\left\|T\right\|_{p,q} = \left\|u\right\|_{q}^{*} \left\|x\right\|_{p}
		$$
		
		The operator $T$ is continuous if and only if for all $k$ there is $N(k)$ such that
		$$
		\left\|T\right\|_{k,N(k)} < \infty,
		$$
		$T$ is bounded if and only if there is $N \in \mathbb{N}$ such that for all $r \in \mathbb{N}$,
		$$
		\left\|T\right\|_{r,N} < \infty.
		$$
		
		We write $(X,Y) \in \mathcal{B}$ if every continuous linear operator on $X$ to $Y$ is bounded. Zahariuta \cite{Zah73} obtained that if the matrices $A$ and $B$ satisfy the conditions $d_{2}$ and $d_{1}$, repectively, then $(\lambda^{l_{1}}(A),\lambda^{l_{1}}(B)) \in \mathcal{B}$. This phenomenon was studied extensively by Vogt \cite{Vog83} not only for K\"{o}the spaces but also for the general case of Fr{\'e}chet spaces. In case of $\ell$-K\"{o}the spaces, there is no characterization of pairs $(X,Y)$ with the property $\mathcal{B}.$ 
		
		For Fr{\'e}chet spaces $X$ and $Y$, in \cite{Vog83}, Vogt proved that $(X,Y) \in \mathcal{B}$ if and only if for every sequence $N(k)$, $\exists N \in \mathbb{N}$ such that $\forall r \in \mathbb{N}$ we have $k_{0} \in \mathbb{N}$ and $C>0$ with
\begin{eqnarray}
\displaystyle {\left\| T \right\|}_{r,N} \leq C \max_{1 \leq k \leq k_{0}} {\left\| T \right\|}_{k,N(k)}
\label{eqn}
\end{eqnarray}
for all $T \in \mathcal{L}(X,Y)$.
		
		An operator $T : \lambda^{\ell}(A) \longrightarrow \lambda^{\ell}(B)$ is called quasi-diagonal if there exists $k: \mathbb{N} \longrightarrow \mathbb{N}$ and constants $m_{n}$ such that 
		$$
		Te_{n} = m_{n} \tilde{e}_{k(n)}, \quad n \in \mathbb{N}
		$$
		
		Following \cite{Kro85}, a pair of K\"{o}the spaces $(\lambda^{\ell}(B),\lambda^{\ell}(A))$ satisfies the condition $\mathcal{S}$ if,
\begin{equation}\label{eqn2}
		\forall p \quad \exists q,k \quad \forall s,l \quad \exists r,C : \frac{b_{m}^{s}}{a_{n}^{k}} \leq C \max \left\{\frac{b_{m}^{q}}{a_{n}^{p}},\frac{b_{m}^{r}}{a_{n}^{l}}\right\}
\end{equation}

		In \cite{Dra72} it was proved that the existence of an unbounded continuous linear operator from nuclear $l_{1}$-K\"{o}the space to another implies the existence of a continuous unbounded quasi-diagonal operator. Also, if the both K\"{o}the spaces are nuclear, in \cite{Nur84}, Nurlu and Terzio\u{g}lu proved that the existence of an unbounded continuous linear operator on $\lambda^{l_{1}}(A)$ to $\lambda^{l_{1}}(B)$ implies, under some conditions, the existence of a common basic subspaces of $\lambda^{l_{1}}(A)$ and $\lambda^{l_{1}}(B)$. Djakov and Ramanujan generalized these results by omitting nuclearity condition \cite{Dja02}.
		
		Let $X = \lambda^{\ell}(A)$ and $Y = \lambda^{\ell}(B)$ be the $\ell$- K\"{o}the spaces. Here, we modify Proposition $1$ in \cite{Dja02} for $\ell$- K\"{o}the spaces and using it we obtain a necessary and sufficient condition in terms of corresponding K\"{o}the matrices when $(X,Y) \in \mathcal{B}$. Also we observe a common basic subspace between $\ell$- K\"{o}the spaces $X$ and $Y$ when $(X,Y) \notin \mathcal{B}$ and $(Y,X) \in \mathcal{S}$ following the same lines in \cite{Dja02}.

\section{Bounded and unbounded operators in $\ell$-K\"{o}the spaces}

    Let $\lambda^{\ell}(A), \lambda^{\ell}(B)$ be $\ell$-K\"{o}the spaces. As in \cite{Dja02} we obtain the following.
\begin{proposition}
Let $\lambda^{\ell}(A)$ and $\lambda^{\ell}(B)$ be $\ell$- K\"{o}the spaces. If there exists a continuous linear unbounded operator $T: \lambda^{\ell}(A) \longrightarrow \lambda^{\ell}(B)$, then there exists a continuous unbounded quasi-diagonal operator on $\lambda^{\ell}(A)$ to $\lambda^{\ell}(B).$
\label{pro}
\end{proposition}
\begin{proof}
Let $T: \lambda^{\ell}(A) \longrightarrow \lambda^{\ell}(B)$ be continuous and unbounded. We may assume without loss of generality that
$$
\left\|Tx\right\|_{k} \leq \frac{1}{2^{k}} \left\|x\right\|_{k}, \quad \forall x \in \lambda^{\ell}(A)
$$
$$
\sup_{n} \frac{\left\| Te_{n} \right\|_{k+1}}{\left\|e_{n}\right\|_{k}} = \infty, \quad k \in \mathbb{N}.
$$
Indeed, one may obtain these by using appropriate multipliers and passing to a subsequence of seminorms, if necessary. Let $(k_{j})$ be a sequence of integers such that each $k \in \mathbb{N}$ appears in it infinitely many times and choose an increasing subsequence $(n_{j})$ such that
$$
\frac{ \|Te_{n_{j}}\|_{k_{j}+1}}{\|e_{n_{j}}\|_{k_{j}}} \geq 2^{j}, \quad \forall j
$$
Let us remind that $\left\|\tilde{e_{v}}\right\|_{k} = b_{v}^{k}$ and $\left\|e_{n}\right\|_{k} = a_{n}^{k}$ and let $\displaystyle Te_{n} = \sum_{v} \theta_{nv} \tilde{e_{v}}.$ Note that,
\begin{align*}
\sup_{\left|\alpha_{v}\right| \leq 1} \left| \sum_{v} \theta_{nv} \alpha_{v} \left(\sup_{k} \frac{b_{v}^{k}}{a_{n}^{k}}\right) \tilde{e_{v}} \right|
&\leq \sum_{k} \left(\frac{b_{v}^{k}}{a_{n}^{k}}\right) \left( \sup_{\left|\alpha_{v}\right| \leq 1} \left| \sum_{v} \theta_{nv} \alpha_{v}  \tilde{e_{v}} \right| \right) \\
&\leq \sum_{k} \frac{1}{{a_{n}^{k}}} \sup_{\left|\alpha_{v}\right| \leq 1} \left| \sum_{v} \theta_{nv} \alpha_{v} {b_{v}^{k}} \tilde{e_{v}} \right| \\
&\leq \sum_{k} \frac{\left\| Te_{n} \right\|_{k}}{\left\|e_{n}\right\|_{k}} \leq \sum_{k} \frac{1}{2^{k}} \leq 1
\end{align*}
Therefore we obtain that
\begin{equation}\label{eqn1}
\sup_{\left|\alpha_{v}\right| \leq 1} \left| \sum_{v} \theta_{n_{j}v} \alpha_{v} \left(\sup_{k} \frac{b_{v}^{k}}{a_{n_{j}}^{k}}\right) \tilde{e_{v}} \right| \leq 1 \leq \frac{1}{2^{j}} \sup_{\left|\alpha_{v}\right| \leq 1} \left| \sum_{v} \theta_{n_{j}v} \alpha_{v} \frac{b_{v}^{k_{j}+1}}{a_{n_{j}}^{k_{j}}} \tilde{e_{v}}\right|
\end{equation}
So there is a $v_{j}$ such that
$$
t_{j}:= \sup_{k} \frac{b_{v_{j}}^{k}}{a_{n_{j}}^{k}} \leq \frac{1}{2^{j}} \frac{b_{v}^{k_{j}+1}}{a_{n_{j}}^{k_{j}}}
$$
Otherwise we obtain a contradiction to \eqref{eqn1} by monotonicity of $\left\|.\right\|.$

    Now, consider the quasi-diagonal operator $D: \lambda^{\ell}(A) \longrightarrow \lambda^{\ell}(B)$ defined by
$$
De_{n_{j}} = t_{j}^{-1} \tilde{e_{v_{j}}} , \quad j \in \mathbb{N}
$$
$$
De_{n} = 0 \quad \text{if} \quad n \neq n_{j}
$$
Let $\displaystyle x = \sum_{j} x_{n_{j}} e_{n_{j}} \in \lambda^{\ell} (A).$ So, $\displaystyle Dx = \sum_{j} x_{n_{j}} t_{j}^{-1} \tilde{e}_{v_{j}}$.
Since $\left|x_{n_{j}} t_{j}^{-1} b_{v_{j}}^{k}\right| \leq \left|x_{n_{j}} a_{n_{j}}^{k}\right|$, by monotonicity we obtain that $\left\|\left(x_{n_{j}} t_{j}^{-1} b_{v_{j}}^{k}\right)\right\| \leq \left\|\left(x_{n_{j}} a_{n_{j}}^{k}\right)\right\|,$ i.e.,
$$ 
\left\|Dx\right\|_{k} \leq \left\|x\right\|_{k} \quad \forall k
$$
Hence, $D$ is continuous.
      
			Similarly, it is easy to see that $D$ is unbounded since for a fixed $k$, there is a subsequence $(j_{m})$ such that $k_{j_{m}} = k$, $m\in \mathbb{N}$ and 
$$
\frac{\|De_{n_{j_{m}}}\|_{k+1}}{\|e_{n_{j_{m}}}\|_{k}} \geq 2^{j_{m}} \rightarrow \infty
$$
as $m \rightarrow \infty.$ This completes the proof.
\end{proof}
Proposition \ref{pro} enables us to prove the sufficiency part of the following theorem. Notice that sufficiency can not be obtained directly for a general linear map.
\begin{theorem}
Let $\lambda^{\ell}(A)$ and $\lambda^{\ell}(B)$ be $\ell$-K\"{o}the spaces. $(\lambda^{\ell}(A),\lambda^{\ell}(B)) \in \mathcal{B}$ if and only if for every sequence $N(k)\uparrow\infty$ there exists $N \in \mathbb{N}$ such that for each $r \in \mathbb{N}$ we have $k_{o} \in \mathbb{N}$ and $C > 0$ with
$$
\frac{b_{v}^{r}}{a_{i}^{N}} \leq C \max_{1 \leq k \leq k_{0}} \frac{b_{v}^{k}}{a_{i}^{N(k)}}
$$
for all $v \in \mathbb{N}, i \in \mathbb{N}$.
\end{theorem}
\begin{proof}
Suppose $(\lambda^{\ell}(A),\lambda^{\ell}(B)) \in \mathcal{B}$. Consider $T: \lambda^{\ell}(A)\longrightarrow \lambda^{\ell}(B)$ with $T = e_{i}' \otimes e_{v}$ where $e_{i}{'}(x) = x_{i}$ for all $x \in \lambda^{\ell}(A).$
    
		Since $T$ is the operator of rank one, we note that
$$
\left\|T\right\|_{k,N(k)} = \left\|e_{i}{'}\right\|_{N(k)} \left\|e_{v}\right\|_{k} = \frac{b_{v}^{k}}{a_{j}^{N(k)}}
$$
Similarly $\left\|T\right\|_{r,N} = \frac{b_{v}^{r}}{a_{i}^{N}}.$ The result follows from \eqref{eqn}.

    Conversely we want to show that every continuous linear quasi-diagonal operator is bounded. Let $T: \lambda^{\ell}(A)\longrightarrow \lambda^{\ell}(B)$ be a continuous  quasi-diagonal operator defined by $T(e_{i}) = t_{i} \tilde{e}_{z(i)}$. By continuity, $\exists N(k)$ such that
$$ 
\sup_{i} \frac{\left\|Te_{i}\right\|_{k}}{\left\|e_{i}\right\|_{N(k)}} = \sup_{i} \frac{\left|t_{i}\right| b_{z(i)}^{k}}{a_{i}^{N(k)}} = C(k) < \infty.
$$
Thus for this $N(k)$, $\exists N \in \mathbb{N}$ such that $\forall r \in \mathbb{N}$ we have $k_{o} \in \mathbb{N}$ and $C > 0$ with
$$
\frac{\left|t_{i}\right| b_{z(i)}^{r}}{a_{i}^{N}} \leq C \max_{1 \leq k \leq k_{0}} \frac{\left|t_{i}\right| b_{z(i)}^{k}}{a_{i}^{N(k)}} \leq C \max_{1 \leq k \leq k_{0}} C(k).
$$
Hence $\left\|T\right\|_{r,N} < \infty$, i.e., $T$ is bounded. In view of Proposition \ref{pro}, we obtain the result.	
\end{proof}

    $\lambda^{\ell}(A)$ and $\lambda^{\ell}(B)$ have a common basic subspace if there is a quasi-diagonal operator $T: X \rightarrow Y$ such that the restriction of $T$ to some infinite dimensional basic subspace of $X$ is an isomorphism. We observe the following extension of Proposition $3$ in \cite{Dja02} to the $\ell$-K\"{o}the space case. The proof is the same as in \cite{Dja02}.
\begin{corollary}
If $(\lambda^{\ell}(B),\lambda^{\ell}(A)) \in \mathcal{S}$ and there exists a continuous unbounded operator $T : \lambda^{\ell}(A) \longrightarrow \lambda^{\ell}(B)$, then $\lambda^{\ell}(A)$ and $\lambda^{\ell}(B)$ have a common basic subspace.
\end{corollary}

\end{document}